\documentclass[11pt,a4paper]{amsart}
\usepackage{amscd,amssymb,amsmath,latexsym}
\usepackage{graphicx,epsfig,color,url}
\usepackage[all]{xy} 
\usepackage[colorlinks,linkcolor=blue,anchorcolor=blue,citecolor=blue]{hyperref}
\usepackage[utf8]{inputenc}  
\usepackage[T1]{fontenc}  
\hypersetup{breaklinks=true}
\def\wt#1{{\widetilde{#1}}}

\newcommand{\Hess}{\operatorname{H}}

\newcommand{\Flex}{\operatorname{Flex}}

\newcommand{\Res}{\operatorname{Res}}

\newcommand{\val}{\operatorname{val}}

\newcommand{\ord}{\operatorname{ord}}

\newcommand{\pr}{\operatorname{pr}}

\renewcommand{\i}{\operatorname{i}}

\renewcommand{\and}{\quad \text{and} \quad}

\newcommand{\A}{\mathbb{A}}

\newcommand{\N}{\mathbb{N}}
\renewcommand{\P}{\mathbb{P}}

\newcommand{\Z}{\mathbb{Z}}

\newcommand{\cL}{{\mathcal L}}

\newcommand{\cO}{{\mathcal O}}

\newcommand{\bfa}{{\boldsymbol{a}}}
\newcommand{\bfb}{{\boldsymbol{b}}}
\newcommand{\bfc}{{\boldsymbol{c}}}
\newcommand{\bfd}{{\boldsymbol{d}}}

\newcommand{\bfp}{{\boldsymbol{p}}}
\newcommand{\bfq}{{\boldsymbol{q}}}

\newcommand{\bfu}{{\boldsymbol{u}}}
\newcommand{\bfv}{{\boldsymbol{v}}}

\newcommand{\bfx}{{\boldsymbol{x}}}
\newcommand{\bfy}{{\boldsymbol{y}}}

\newcommand{\bfalpha}{\boldsymbol{\alpha}}
\newcommand{\bfbeta}{\boldsymbol{\beta}}

\newcommand{\bfF}{\boldsymbol{F}}

\newcommand{\bfzero}{{\boldsymbol{0}}}

\numberwithin{equation}{section}

\theoremstyle{definition}
\newtheorem{definition}{Definition}[section]

\newtheorem{remark}[definition]{Remark}
\newtheorem{example}[definition]{Example}

\theoremstyle{plain}
\newtheorem{lemma}[definition]{Lemma}
\newtheorem{proposition}[definition]{Proposition}
\newtheorem{theorem}[definition]{Theorem}
\newtheorem{corollary}[definition]{Corollary}

\begin{document}

\title[The  flex locus of a hypersurface]{The geometry of the flex locus of a hypersurface}

\author[Bus\'e]{Laurent Bus\'e}
\address{INRIA Sophia Antipolis. 2004 route des Lucioles, 06902 Sophia Antipolis, France}
\email{Laurent.Buse@inria.fr }
\urladdr{\url{http://www-sop.inria.fr/members/Laurent.Buse/}}

\author[D'Andrea]{Carlos D'Andrea} 
\address{Departament de Matem\`atiques i
  Inform\`atica, Universitat de Barcelona. Gran Via 585, 08007
  Barcelona Spain} 
\email{cdandrea@ub.edu}
\urladdr{\url{http://www.ub.edu/arcades/cdandrea.html}}

\author[Sombra]{Mart{\'\i}n~Sombra} 
\address{Instituci\'o Catalana de Recerca
  i Estudis Avan\c{c}ats (ICREA). Passeig Llu{\'\i}s Companys~23,
  08010 Barcelona, Spain  \vspace*{-2.5mm}} 
\address{Departament de Matem\`atiques i
  Inform\`atica, Universitat de Barcelona. Gran Via 585, 08007
  Bar\-ce\-lo\-na, Spain} 
\email{sombra@ub.edu}
\urladdr{\url{http://www.maia.ub.edu/~sombra}}

\author[Weimann]{Martin Weimann}
\address{Laboratoire de math\'ematiques Nicolas Oresme, UMR CNRS 6139,
  Universit\'e de Caen. BP 5186, 14032 Caen Cedex,
  France \vspace*{-2.5mm}}
\address{Laboratoire de math\'ematiques GAATI, University of French
  Polynesia. BP 6570, 98702 Faaa,
  French Polynesia}
\email{weimann@unicaen.fr    }
\urladdr{\url{https://weimann.users.lmno.cnrs.fr/}}

\date{\today} 

\begin{abstract}
  We give a formula in terms of multidimensional resultants for an
  equation for the flex locus of a projective hypersurface,
  generalizing a classical result of Salmon for surfaces in $\P^{3}$.
  Using this formula, we compute the dimension of this flex locus, and
  an upper bound for the degree of its defining equations.
  We also show that, when the hypersurface is generic, this bound is reached, and that the generic flex line is
  unique and has the expected order of contact with the hypersurface.
\end{abstract}
\maketitle

\section{Introduction}

A point of a projective variety is a flex point if there is a line
with order of contact with the variety at this point higher
than expected. It is a generalization of the notion of inflexion point
of a curve. The study of the flex locus of curves and surfaces is a
classical subject of geometry from the XIXth century, treated by
Monge, Salmon and Cayley,  among
others. Currently, there is an increasing
interest in this object in low dimensions, mainly due to
its applications in incidence geometry \cite{Tao:MCStcdg,
  Kat14, GuthKatz:eddpp, Kollar:stttd3, EH, SharirSolomon:ibplttdv,GZ18}.

In this text, we study the geometry of the flex locus of a
hypersurface of a projective space of arbitrary dimension.  Before
explaining our results, we introduce some notation and summarize the
previous. Let $ K$ be an algebraically closed field of characteristic
zero, $\P^n$ the  projective space over~$K$ of dimension $n\ge 1$, and $V$ a
hypersurface of $\P^{n}$ of degree $d\ge 1$.  A point $p\in V$ is a
\emph{flex point} if there is a line with order of contact at least
$n+1$ with the hypersurface $V$ at the point $p$, and any such line is
called a \emph{flex line} (Definition \ref{def:2}). The \emph{flex locus} of $V$ is the set of all the flex points of $V$.  

An important result in this context is the so-called
Monge-Salmon-Cayley theorem for surfaces in $\P^{3}$, see for instance
\cite{Tao:MCStcdg, Kollar:stttd3}, generalized by Landsberg to the
higher dimensional case \cite[Theorem 3]{Land}. It states that if the
hypersurface $V$ is irreducible, then it is ruled if and only if all
of its points are flexes.

A hypersurface of degree less than $n$ is necessarily ruled
(Proposition \ref{lmu}) and its flex locus is the whole
hypersurface. Hence, one restricts the study of the flex locus to the
case $d\ge n$.

For a plane curve $C\subset \P^{2}$ of degree $d\ge 2$, a point
$p\in C$ is an inflexion point if and only if the determinant of the
Hessian matrix of the defining polynomial of $C$ vanishes at
$p$. This implies that the flex locus of $C$ is defined by a
polynomial of degree $3d-6$. Hence if $C$ contains no line, then it
has at most $3d^2-6d$ inflexion points, by B\'ezout theorem.

For a surface $S\subset \P^{3}$ of degree $d\ge 3$, an old result of
Salmon  states that there is a homogeneous polynomial in
$ K[x_{0},x_{1},x_{2},x_{3}]$ of degree $11d-24$ defining its flex
locus \cite[Article 588, pages 277--278]{Salmon:tagtd}, see also
\cite[\S 11.2.1]{EH}.  If $S$ has no ruled component, then this
result together with the Monge-Salmon-Cayley theorem and B\'ezout's
theorem imply that the flex locus is a curve of $S$ of degree at most
$  11d^2-24d.$

\medskip

We first address the problem of computing the dimension, and the
degree of both the defining equations and the flex locus.  Let
$\bfx=\{x_{0},\dots, x_{n}\}$ be a set of $n+1$ variables and
$f_{V}\in K[\bfx]$ a squarefree homogeneous polynomial defining
$V$. Let $t$ be another variable and $\bfy=\{y_{0},\dots, y_{n}\}$ a
further set of $n+1$ variables. Then we consider the family of
bihomogeneous polynomials $f_{V,k}$, $k=0,\dots, d$, in $K[\bfx,\bfy]$
determined by the expansion
\begin{equation*}
f_{V}(\bfx+t\bfy) =\sum_{k=0}^d f_{V,k}(\bfx,\bfy)\frac{t^k}{k!}.
\end{equation*}

Our first main result gives an equation for the flex locus of $V$ in
terms of multivariate resultants.

\begin{theorem}
  \label{thm:2}
There is a homogeneous polynomial $\rho_{V}\in K[\bfx]$ with
  \begin{equation*}
      \deg(\rho_{V}) =  d\sum_{k=1}^{n}\frac{n!}{k}-(n+1)!
  \end{equation*}
  defining the flex locus of $V$. It is uniquely determined modulo
  $f_{V}$ by the condition
\begin{equation*}
  \Res^{\bfy}(f_{V,1}(\bfx,\bfy), \dots, f_{V,n}(\bfx,\bfy), \ell(\bfy))
\equiv \ell^{n! } \rho_{V} \mod{f_V}, 
\end{equation*}
for any linear form $\ell \in K[\bfx]$,
where $\Res^{\bfy}$ denotes the resultant of $n+1$ homogeneous
polynomials in the variables $\bfy$.  
\end{theorem}

This result recovers the previous degree computations for the
polynomial defining the flex locus of a plane curve or of a surface in
$\P^{3}$.  It also allows us to give a scheme structure to the flex
locus: we define the \emph{flex scheme} $\Flex(V) $ as the {subscheme}
of $\P^{n}$ defined by the homogeneous polynomials $f_{V}$ and
$\rho_{V}$ (Definition \ref{def:3}). This scheme does not depend on
the choice $f_{V}$, unique up to a nonzero scalar factor, nor on that
of $\rho_{V}$, unique modulo $f_{V}$. Thus, the flex locus of $V$ is
the reduced scheme associated to $\Flex(V)$.

The next corollary is a direct consequence of Theorem \ref{thm:2} and
Landsberg's theorem generalizing the Monge-Salmon Cayley theorem
\cite[Theorem 3]{Land}.

\begin{corollary}
  \label{cor:1} If $V$ has no ruled irreducible components, then
  $\Flex( V) $ is a complete intersection subscheme of $\P^{n}$ of dimension $n-2$ and of
  degree
\begin{equation*}
\deg (\Flex(V)) =
d^2\sum_{k=1}^{n}\frac{n!}{k}-d\, (n+1)!.
\end{equation*} 
In particular, the flex locus of $V$ is set-theoretically defined by
equations of degree at most
$\max(d,d\sum_{k=1}^{n}\frac{n!}{k}-(n+1)!)$, and its degree, as an
algebraic set, is at most $d^2\sum_{k=1}^{n}\frac{n!}{k}-d(n+1)!$.
\end{corollary}

Set $\cL_{V}$ for { the union of the lines} contained in $V$.  When $d=n$, a
flex line of $V$ at a point $p\in V$ has order of contact at least
$n+1$ at this point, and so it is necessarily contained in $V$ by B\'ezout theorem. Hence
in this case, $\cL_{V}$ coincides with the flex locus of $V$.

\begin{corollary}
  \label{cor:2}
Let $V$ be a hypersurface of $\P^{n}$ of degree $n$ without  ruled irreducible
component. Then $\cL_{V}$ is a
ruled subvariety of $V$ of dimension $n-2$
and of degree
  at most 
    \begin{displaymath}
  n^{3}\, (n-1)! \, \sum_{k=2}^{n-1} \frac{1}{k}.
\end{displaymath}
\end{corollary}

Our second main result ensures that the bound for the degree of the
flex locus is sharp, and that other expected properties hold true in
the generic case. 
These properties of generic hypersurfaces and flex are proven using resultant theory. These aspects were not considered in the original work of Salmon, and the obtained results are new in every dimension.

\begin{theorem}\label{tmain2}
Let $V$ be a generic hypersurface  of $\P^{n}$ of degree $d\ge n$. Then
\begin{enumerate}
\item \label{item:3} $\Flex(V)$ is a reduced subscheme (that is, a
  subvariety) of $V$ of dimension $n-2$;
\item \label{item:4} for a generic flex point $p$ of $V$, there is a
  unique flex line { containing} it. If $d=n$, then this line is
  contained in $V$, whereas if $d>n$, then its order of contact with $V$ at
  $p$ is exactly $n+1$.
\end{enumerate}
\end{theorem}

For a cubic surface $S$ in $\P^{3}$, Salmon's degree bound is
$11\cdot 3-24 =9$. If $S$ is smooth, it contains $27$ lines and their
union is the complete intersection of $S$ with a surface of degree
$9$. The next result gives an analogous result for generic
hypersurfaces of $\P^{n}$ of degree $n$.  It is a direct consequence
of Theorem \ref{tmain2} and Corollary \ref{cor:2}.

\begin{corollary}
  \label{cor:22}
  Let $V$ be a generic hypersurface of $\P^{n}$ of degree $n$. Then
  $\cL_{V}$ is a ruled subvariety of $V$ of dimension $n-2$ of degree
  equal to 
    \begin{displaymath}
  n^{3}\, (n-1)! \, \sum_{k=2}^{n-1} \frac{1}{k},
\end{displaymath}
complete intersection of $V$ with a hypersurface of degree
$ n^{2}\, (n-1)! \, \sum_{k=2}^{n-1} \frac{1}{k}$.
\end{corollary}

Salmon's theorem for surfaces has been revisited several times. In
particular, in the recent book \cite{EH}, the authors reprove it by
performing suitable computations in the Chow ring of a Grassmaniann.
Our proof of this result proceeds by identifying lines with points
of $\P^n\times\P^n$ outside the diagonal, in the spirit of Salmon's original approach; 
see \cite[Articles 473 and 588, pages 94--95 and 277--278]{Salmon:tagtd} and Remark \ref{rem:2}.  
Although this seems less
natural from the point of view of intersection theory, it nevertheless
allows us to find explicit equations for the flex locus using
resultants and to prove it in a more general setting; see Theorem
\ref{thm:2}.

\medskip

The paper is organized in the following way. In Section \ref{section2}
we review the definition and properties of multidimensional resultants
that will be used in the sequel. The proof of Theorem \ref{thm:2} is
given in Section \ref{sec:affineflex}, whereas in Section
\ref{section4} and Section \ref{section6} we show that the flex
subscheme is generically reduced and that the generic flex line is
unique and has the expected order of contact, thus proving Theorem
\ref{tmain2}.

\medskip
\paragraph{\bf Acknowledgements} 
We thank Noam Solomon for several illuminating discussions on Salmon's
approach to the flex polynomial.  We also thank Marc Chardin, Mart\'i
Lahoz, Juan Carlos Naranjo and Patrice Philippon for other helpful
discussions. Part of this work was done while the authors met at the
Universitat de Barcelona, the Université de Nice - Sophia Antipolis, and
the Université de Caen.  We thank these institutions for their
hospitality.

Bus\'e and Weimann were partially supported by the CNRS research project PICS 6381 ``Diophantine
geometry and computer algebra''.  D'Andrea and Sombra were partially
supported by the MINECO research project MTM2015-65361-P, and by the
``Mar\'ia de Maeztu'' program for units of excellence in R\&D
MDM-2014-0445.

\section{Preliminaries on resultants} \label{section2}

The resultant of a family of homogeneous multivariate polynomials
plays a central role throughout this text. Therefore, in this section
we briefly review this notion and some of its basic properties.  We
refer to \cite{CLO98,Jou,GKZ94} for the proofs and more details.

We denote by $\N$ the set of nonnegative integers and by $K$ an
algebraically closed field of characteristic zero.  Boldface symbols
indicate finite sets or sequences, where the type and number should be
clear from the context. For instance, for $n\in \N$ we denote by
$\bfy$ the set of variables $\{y_{0},\dots, y_{n}\}$, so that
$K[\bfy]=K[y_{0},\dots, y_{n}]$.

Let $\bfd=(d_{0},\dots,d_{n})\in \N^{n+1}$. For  
$i=0,\dots, n$, we consider the general homogeneous polynomial of degree $d_{i}$
in the variables $\bfy$ given by 
$$
F_i=\sum_{|\bfa|=d_i}c_{i,\bfa}\bfy^{\bfa},
$$
the sum being over the vectors of
$\bfa=(a_{0},\dots, a_{n}) \in \N^{n+1} $ of length
$|\bfa|=\sum_{j=0}^{n}a_{j}=d_{i}$, and where each $c_{i,\bfa}$ is a
variable and $\bfy^{\bfa}$ stands for the monomial
$\prod_{j=0}^{n}y_{j}^{a_{j}}$.

For each $i$, set
$\bfc_{i}=\{c_{i,\bfa} \mid \bfa\in \N^{n+1} , |\bfa|=d_{i}\}$ for the
set of ${d_{i}+n\choose n}$ variables corresponding to the
coefficients of $F_{i}$, and $A=\Z[\bfc_{0},\dots, \bfc_{n}]$ for the
universal ring of coefficients. As usual, given $P\in A$ and a system
of homogeneous polynomials $g_{i}\in K[\bfy]$ of degree $d_{i}$,
$i=0,\dots, n$, we write
\begin{displaymath}
P(g_{0},\dots, g_{n}) \in K
\end{displaymath}
for the evaluation of $P$ in the coefficients of the $g_{i}$'s. 

Denote by $I$ and by $\mathfrak{m}$ the ideals of $A[\bfy]$
respectively defined by $F_{0},\dots, F_{n}$ and by
$y_{0},\dots, y_{n}$.  The {\it elimination ideal} of the system  $\bfF=(F_{0},\dots, F_{n})$ is the ideal of
$A$ defined by
\begin{displaymath}
  E_{\bfd}=\{P\in A \mid \exists k\in \N \text{ with }
  P \, \mathfrak{m}^{k}\subset I\}.
\end{displaymath}
It is a principal ideal, and the \emph{resultant} of $\bfF$, denoted
by $\Res_{\bfd}$, is defined as its unique generator satisfying the
additional condition
  \begin{equation*}
 \Res_\bfd(y_0^{d_0},\ldots,y_n^{d_n})=1.
 \end{equation*}
 It is an irreducible polynomial in the ring $A$ that is homogeneous
 of degree $ \prod_{j\ne i} d_{j}$ in each  set of variables
 $\bfc_{i}$, $i=0,\ldots,n$.

The resultant also verifies the following formula for the {descent of
  dimension} \cite[Lemme 4.8.9 and \S 5.7]{Jou}. 

 \begin{proposition}
   \label{prop:2}
With notation as above, 
\begin{displaymath}
\Res_{(d_{0},\dots, d_{n-1},d_{n})}(F_{0},\dots, F_{n-1},y_{n}^{d_{n}})=     \Res_{(d_{0},\dots, d_{n-1})}^{d_{n}},
 \end{displaymath}
 where $ \Res_{(d_{0},\dots, d_{n-1})}$ denotes the resultant of $n$
 general homogeneous polynomials in $A[y_{0},\dots, y_{n-1}]$ of
 respective degrees $d_{0},\dots, d_{n-1}$.
\end{proposition}

The resultant satisfies the \emph{Poisson formula} that we state
below, see \cite[Proposition 2.7]{Jou} or \cite[Theorem 3.4, Chapter 3]{CLO98} for its proof. 

\begin{proposition} \label{prop:8} Let $g_{0},g_{0}'\in K[\bfy]$ be
  homogeneous polynomials of degree $d_{0}$, and
  $g_1,\dots, g_{n} \in K[\bfy]$ homogeneous polynomials of respective
  degrees $d_{1},\dots, d_{n}$ with a finite number of common zeros in
  $\P^{n}$.  For each common zero $\eta \in \P^{n}$ of $g_1,\dots, g_{n}$, let
  $m_{\eta}$ denote its multiplicity. Then
\begin{equation*}
  \label{eq:13}
\Res_{\bfd}(g_0,g_{1},\ldots,g_n) \prod_{\eta}{g_0'(\eta)^{m_{\eta}}} = \Res_{\bfd}(g_0',g_{1},\ldots,g_n)
 \prod_{\eta}{g_0(\eta)^{m_{\eta}}},
\end{equation*}
both products being over the set of common zeros of $g_1,\ldots,g_n$
in $\P^{n}$.
 \end{proposition}

 A fundamental property of resultants is that their vanishing
 characterizes the systems of $n+1$ homogeneous polynomials in $n+1$
 variables that are degenerate, in the sense that their zero set in
 $\P^{n}$ is nonempty. Precisely, a system of homogeneous polynomials
 $g_{0},\dots, g_{n}\in K[\bfy]$ of respective degrees
 $d_{0}, \dots, d_{n}$, has a common zero in $\P^{n}$ if and only if
 $\Res_{\bfd}(g_{0},\dots, ,g_{n})=0$.

 The following result gives a criterion to decide if a such degenerate
 system has a unique zero and, if it does, allows to compute it, see
 \cite[Lemma 4.6.1]{Jou} or \cite[Corollary
 4.7]{JeronimoKrickSabiaSombra:cccf} for its proof.

\begin{proposition}
  \label{prop:9}
  Let $g_{0},\dots, g_{n}\in K[\bfy]$ be homogeneous polynomials of
  respective degrees $d_{0},\dots, d_{n}$. Suppose that
  $\Res_{\bfd}(g_{0},\dots, ,g_{n})=0$ and that there is $0\le i_{0}\le n$
  and $\bfa_{0}\in \N^{n+1}$ with $|\bfa_{0}|=d_{i_{0}}$ such that
  \begin{displaymath}
    \frac{\partial \Res_{\bfd}}{\partial c_{i_{0},\bfa_{0}}} (g_{0},\dots,
  ,g_{n})\ne 0.
  \end{displaymath}
  Then the zero set of $g_{0},\dots, g_{n}$ in $\P^{n}$ consists of a
  single point $\eta$, and for $i=0,\dots, n$, 
\begin{equation*}
(\eta^{\bfa})_{|\bfa|=d_{i}} = \Big( \frac{\partial \Res_{\bfd}}{\partial
c_{i,\bfa}} (g_{0},\dots,
  ,g_{n})\Big) _{|\bfa|=d_{i}} \in \P^{{d_{i}+n\choose n}}
\end{equation*}
{ where the coordinates of these projective points are indexed by the
vectors $\bfa \in \N^{n+1}$ with $|\bfa| =d_{i}$.}
\end{proposition}

\section{The equation of the flex locus}\label{sec:affineflex}

In this section, we obtain an explicit equation 
for the flex locus of a projective hypersurface by means of resultants. Using this equation,
we define the flex scheme and  we compute its dimension, the degree of its 
defining equations and its degree, thus giving the proof of Theorem \ref{thm:2}.

\begin{definition}
  \label{def:1}
  Let $V$ be a subvariety of $\P^n$ and  $p$ a point $V$. For a line 
  $L$  of $ \P^n$  {containing} $p$, its  \emph{order of contact} with $V$
  at $p$ is defined as
\begin{equation*}
\ord_p(V,L)=\dim_{ K}(\mathcal{O}_{L,p}/\iota^*\mathcal{I}_{V}),
\end{equation*} 
where $\mathcal{O}_{L,p}$ is the local ring of $L$ at $p$,
$\mathcal{I}_V$  the ideal sheaf of $V$, and
$\iota\colon L\hookrightarrow \P^n$ the inclusion map.
\end{definition}

The order of contact of a line is either a positive integer or
$+\infty$.  We have that $\ord_p(V,L)= 1$ if and only if $L$
intersects $V$ transversally at~$p$, and $\ord_p(V,L)=+\infty$ if and
only if $L$ is contained in $V$.

For the rest of this section, we assume that $V$ is a (non necessarily
irreducible) hypersurface of degree $d\ge1$. Fix then a defining
polynomial $f_{V}$ of $V$, that is, a homogeneous polynomial in
$K[\bfx]= K[x_{0},\dots, x_{n}]$ of minimal degree such that $V$
coincides with $Z(f_{V})$, the set of zeros of $f_{V}$ in $\P^{n}$.  Such a
polynomial is squarefree, and unique up to a nonzero scalar factor.

The next lemma translates the notion of order of contact with the
hypersurface $V$ into algebraic terms. Given a variable $t$, we denote
by $\val_t$ the $t$-adic valuation in the local ring
$ K[t]_{(0)}\simeq \cO_{\A^{1},0}${: for $h \in K[t]_{(0)}$ written as
$h_{1}/h_{2}$ with $h_{1},h_{2}\in K[t]$ and $h_{2}$ not vanishing at
the point $0$,  $\val_{t}(h)$ is defined as the least exponent
appearing in the nonzero monomials of the polynomial $h_{1}$.}

\begin{lemma}\label{lkey}
  Let $p\in V$ and $L$ a line of $\P^{n}$ containing $p$. Let
  $\varphi\colon \A^{1}\to \P^{n}$ be an affine map parameterizing a
  neighborhood of $p$ in $L$, and such that $\varphi(0)=p$. Write
  $\varphi=(\ell_{0},\dots, \ell_{n})$ with $\ell_{i}\in K[t]$ an
  affine polynomial, $i=0,\dots, n$.  Then
$$
\ord_p(V,L)=\val_t (f_{V} (\ell_{0},\dots, \ell_{n})).
$$
\end{lemma}

\begin{proof}
  Up to a reordering of the homogeneous coordinates of $\P^{n}$, we
  can suppose that $\ell_0(0)\ne 0$. Let $\wt f_V$ denote the
  dehomogenization of $f_{V}$ in the chart $(x_0\ne 0)\simeq
  \A^{n}$. Using the relation
$$
f_{V}=x_0^{d} \, \wt f_V\Big(\frac{x_1}{x_0},\ldots,\frac{x_n}{x_0}\Big)
$$
and  the fact that  $\val_t (\ell_0) = 0$, we obtain
\begin{equation*}
\val_t (f_{V} (\ell_{0},\dots, \ell_{n})) = \val_t
\Big(\wt f_V\Big(\frac{\ell_1}{\ell_0},\ldots,\frac{\ell_n}{\ell_0}\Big)\Big)
= \ord_p(V,L),
\end{equation*}
where the second equality follows from the definition of the order of
contact, and the fact that $\wt f_V$ is a local equation for the germ of hypersurface
$(V,p)$ and $\varphi$ a parametrization of the germ of line $(L,p)$.
\end{proof}

\begin{definition}
  \label{def:2}
  Let $p\in  V$. The \textit{order of osculation} of $V$
  at $p$ is defined as
\begin{equation*}
\mu_p(V)=\sup_{L} \, \ord_p(V,L),
\end{equation*} 
where the supremum is taken over the lines $L$ of $\P^n$ {
    containing} $p$. The point $p$ is a \emph{flex point} of $V$
  whenever $\mu_p(V)\ge n+1$. A line $L$ with order of contact with
  $V$ at $p$ at least $n+1$ is called a \emph{flex line}.
\end{definition}

Consider again the group of variables $\bfy=\{y_{0},\dots, y_{n}\}$
and a further variable $t$,
and let $f_{V,k}$, $k=0,\dots, d$, be the family of polynomials in
$K[\bfx,\bfy]$ determined by the expansion
\begin{equation} \label{eq:14}
f_{V}(\bfx+t\bfy) =\sum_{k=0}^d f_{V,k}(\bfx,\bfy)\frac{t^k}{k!}.
\end{equation}
For $k=0,\dots, d$, 
\begin{equation*}
  f_{V,k}(\bfx,\bfy)=\sum_{0\le i_1,\ldots, i_k\le n} \frac{\partial^{k}f}{\partial
    x_{i_{1}}\cdots \partial x_{i_k}}(\bfx) \, y_{i_1}\cdots y_{i_k} .
\end{equation*}
In particular, $f_{V,k}$ is bihomogeneous of bidegree $(d-k,k)$.

For a point $p\in \P^n$ and each $k\in \N$,  consider the
subvariety of $\P^{n}$ defined as
\begin{equation*}
Z_{p}^k=\{q \in \P^n \mid f_{V,1}(p,q)=\cdots =f_{V,k}(p,q)=0\}.  
\end{equation*}
The next lemma shows that the order of osculation of $V$ at $p$ can be
read from the dimensions of these subvarieties.

\begin{lemma}\label{lmain} 
Let $p\in V$. 
\begin{enumerate}
\item \label{item:1} For each $k \in \N$, the subvariety
  $Z_{p}^k\subset \P^n$ is a cone centered at $p$, union of the lines
  having  order of contact  with $V$ at $p$ greater than $k$.
\item \label{item:2} The order of osculation of $V$ at $p$ is the
  least $k\in \N$ such that $ Z_{p}^k = \{p\}$.
\end{enumerate}
\end{lemma}

\begin{proof} 
Fix $k\in \N$ and choose a representative $\bfp\in K^{n+1}\setminus
\{\bfzero\}$ of  the
  point $p \in \P^{n}$. We have that   $f_{V}(\bfp+t\bfp)=(1+t)^d f_{V}(\bfp)$ and so, for all
  $j \in \N$, 
  \begin{displaymath}
    f_{V,j}(\bfp,\bfp)=d(d-1)\cdots (d-j) f_{V}(\bfp)=0.
  \end{displaymath}
Hence $p\in Z_{p}^{k}$.

Let $q\in \P^{n}$ be a point different from $p$ and $L$ { the unique line
containing} $p$ and $q$. This line is parametrized by the affine map
$\varphi\colon \A^{1}\to \P^{n}$ defined by
$\varphi(t)= \bfp+ \bfq\, t$ for any choice of representatives
$\bfp, \bfq\in K^{n+1}\setminus \{\bfzero\}$ of $p$ and $q$.
Lemma~\ref{lkey} combined with the expansion \eqref{eq:14} implies
that the condition $q\in Z_{p}^k $ is equivalent to
$ \ord_{p}(V,L) > k$.  Hence $q$ lies in $Z_{p}^k $ if and only if the
line $L$ is contained in this subvariety and has order of contact with
$V$ at $p$ greater than $k$, which proves \eqref{item:1}.

By definition, the order of osculation of $V$ at $p$ is the least
$k\in \N$ such that there is no line $L$ with $ \ord_{p}(V,L) >
k$. Hence \eqref{item:2} is a consequence of \eqref{item:1}.
\end{proof}

The following corollary follows directly from Lemma \ref{lmain} and
the definition of flex points.

\begin{corollary}\label{cor:regular} 
A point $p\in V$ is a flex point if and only if $Z_p^n \ne \{p\}$
  or, equivalently, if and only if $ \dim (Z_p^n)\ge 1$.
\end{corollary}

{ The next two propositions are classical, and they are
  also consequences of Lemma~\ref{lmain}. We include their proofs for
  lack of a suitable reference. }

\begin{proposition}\label{lmu}
  Let $p\in V$. Then either $n\le \mu_p(V)\le d$ or there is a line
{  containing}  $p$ and  contained in $V$. In particular, every
  hypersurface of degree at most $ n-1$ is ruled.
\end{proposition}

\begin{proof}
  For $k \in\N$, the subvariety $Z_{p}^{k}$ is defined by $k$
  equations. If this subvariety consists of the single point $p$, then
  this number of equations $k$ has to be at least $n$, by Krull's
  Hauptidealsatz.  Lemma \ref{lmain}\eqref{item:2} then gives the
  lower bound $\mu_p(V)\ge n$.

  On the other hand, if $k>d$ then $Z_{p}^{k}=Z_{p}^{n}$ because
  $f_{V,j}=0$ for all $j>d$. Hence the $Z_{p}^{k}$'s form a sequence
  of subvarieties that is decreasing with respect to the inclusion,
  and constant for $k\ge d$.  By Lemma \ref{lmain}\eqref{item:2}, if
  $Z^{d}_{p}=\{p\}$ then $\mu_p(V)\le d$. Else, by Lemma
  \ref{lmain}\eqref{item:1}, each line contained in $Z^{d}_{p}$ has an
  order of contact that is arbitrarily large. By Lemma \ref{lkey},
  such a line is necessarily contained in $V$.

To conclude, we observe that the last statement is a direct consequence of the first one.
\end{proof}

\begin{proposition}
  \label{prop:5}
  Every singular point of $V$ is a flex point.
\end{proposition}

\begin{proof}
  A point $p\in V$ is singular if and only if
  $f_{V,1}(p,\bfy)=0$. Hence $Z_{p}^{n}$ is defined by $n-1$
  equations. Since this subvariety contains $p$, it is nonempty and
  so, by Krull's Hauptidealsatz, its dimension is at least 1. By
  Corollary \ref{cor:regular}, this point is necessarily a flex.
\end{proof}

For a homogeneous polynomial $g\in K[\bfx]$ of degree $e\ge 1$, we set
\begin{equation} \label{eq:2}
R_{V,g}=\Res_{(1, \dots, n, e)}^{\bfy}(f_{V,1}(\bfx,\bfy),\ldots,f_{V,n}(\bfx,\bfy),g(\bfy)) \in K[\bfx],
\end{equation}
where $\Res_{(1, \dots, n, e)}^{\bfy}$ denotes the resultant of $n+1$
homogeneous polynomials in the variables $\bfy$ of respective degrees
$1, \dots, n, e$.

\begin{proposition}\label{cor:flex}
  Let $g\in K[\bfx]$ be a homogeneous polynomial of degree $e\ge 1$.
  Then $R_{V,g}$ defines the flex locus of $V$ in the open subset $\P^{n}\setminus Z(g)$.
\end{proposition}

\begin{proof}
  Let $p \in V$ such that $g(p)\ne 0$.  If $R_{V,g}(p)=0$ then $Z(g)$
  intersects $Z_p^n$, by the vanishing property of the
  resultant. Since $p\notin Z(g)$, this implies that $Z_p^n\ne\{p\}$.
  By Corollary \ref{cor:regular}, $p$ is a flex point.
  Conversely, suppose that $p$ is a flex point. Since $g$ is not a
  constant, $Z(g)$ is a hypersurface and, by
  Corollary \ref{cor:regular}, $\dim(Z_p^n)\ge 1$. Hence $Z(g)$ does
  intersect $Z_{p}^{n}$ and so $R_{V,g}(p)=0$, as stated.
\end{proof}

The polynomial $R_{V,g}$ gives an equation for the flex locus of $V$
outside the hypersurface $Z(g)$, but might vanish at points in $Z(g)$
that are not flexes.  The next result, corresponding to Theorem
\ref{thm:2} in the introduction, shows that this equation can be
replaced by another one defining the flex locus of $V$ in the whole of
the projective space.

\begin{theorem}
\label{thm:4}
  There exists a homogeneous polynomial $\rho_{V}\in K[\bfx]$ with
  \begin{equation*}
      \deg(\rho_{V}) =  d\sum_{k=1}^{n}\frac{n!}{k}-(n+1)!
  \end{equation*}
  defining the flex locus of $V$. It is uniquely determined modulo
  $f_{V}$ by the condition that, for any linear form
  $\ell \in K[\bfx]$,
\begin{equation} \label{eq:19}
  R_{V,\ell} \equiv \ell^{n! } \rho_{V} \mod{f_{V}}.
\end{equation}
\end{theorem}

To prove it, we need the following auxiliary result.

\begin{lemma}
\label{lemm:2}
Let $g,h\in K[\bfx]$ be two homogeneous polynomials of the same
positive degree. Then
\begin{displaymath}
   h^{n!} R_{V,g}  \equiv g^{n!} R_{V,h}  \mod{f_{V}}.
\end{displaymath}
\end{lemma}

\begin{proof}
  Let $p\in V$ and $\bfp\in K^{n+1}\setminus\{\bfzero\}$ a
  representative of this point. If $p$ is not a flex, then
  $Z_{p}^{n}=\{p\}$ by Corollary \ref{cor:regular}. By B\'ezout's
  theorem, the intersection multiplicity of
  $f_{V,1}(\bfp,\bfy),\dots, f_{V,n}(\bfp,\bfy)$ at $p$ is $n!$ and hence, by the
  Poisson formula (Proposition \ref{eq:13}),
  \begin{equation} \label{eq:4} 
h(\bfp)^{n!}  R_{V,g}(\bfp) = g(\bfp)^{n!} R_{V,h}(\bfp) .
\end{equation}
On the other hand, if $p$ is flex then $Z_{p}^{n}$ has positive
dimension, again by Corollary~\ref{cor:regular}. This implies that the
system $f_{V,1}(\bfp,\bfy),\dots, f_{V,n}(\bfp,\bfy),G(\bfy)$ has a common zero
and so $R_{V,g}(p)=0$ and, similarly $R_{V,h}(p)=0$. Hence
\eqref{eq:4} reduces to $0=0$ in this case.
Thus the equality \eqref{eq:4} holds for every point of $V$, which
implies the statement.
\end{proof}

\begin{proof}[Proof of Theorem \ref{thm:4}]
  Let $\bfu=\{u_{0},\dots, u_{n}\}$ and $\bfv=\{v_{0},\dots, v_{n}\}$ be
  two sets of $n+1$ variables and consider the  linear forms
\begin{displaymath}
\ell_{\bfu}=\sum_{i=0}^{n} u_{i}x_{i} \quad  \text{ and  } \quad 
\ell_{\bfv}=\sum_{i=0}^{n} v_{i}x_{i} .
\end{displaymath}
 By Lemma \ref{lemm:2}, for every choice of $\bfalpha,\bfbeta\in
 K^{n+1}\setminus \{\bfzero\}$,
\begin{equation*}
  R_{V,\ell_{\bfalpha}(\bfx)}   \, \ell_{\bfbeta}(\bfx)^{n! } \equiv
  R_{F,\ell_{\bfbeta}(\bfx)} \, \ell_{\bfalpha}(\bfx)^{n! }  \mod{f_{V}}.
\end{equation*}
We deduce that there is a trihomogeneous polynomial
$s \in K[\bfu,\bfv,\bfx]$ such that
\begin{equation}\label{eq:18}
   R_{V,\ell_{\bfu}}   \ell_{\bfv}^{n! } -  R_{V,\ell_{\bfv}}
   \ell_{\bfu}^{n! } + s  f_{V} =0. 
\end{equation}
The polynomials $\ell_{\bfu}^{n!},\ell_{\bfv}^{n!}, F$ form a regular
sequence in $ K[\bfu,\bfv,\bfx]$, and hence the syzygy \eqref{eq:18}
is necessarily a Koszul syzygy. Hence there are trihomogeneous
polynomials $\rho_{V},\sigma\in K[\bfu,\bfv,\bfx]$ such that
\begin{equation*}
   R_{V,\ell_{\bfu}}= \ell_{\bfu}^{n!} \, \rho_{V} + f_{V} \sigma.
\end{equation*}
Since  $ \deg_{\bfu} (R_{V,\ell_{\bfu}})=n!$ and $\deg_{\bfv}
(R_{V,\ell_{\bfu}})=0$, we deduce that $\rho_{V}\in K[\bfx]$.  The equality
\eqref{eq:19} is obtained by specializing the variables $\bfu$ into
the coefficients of the linear form~$\ell$. 

By this equality \eqref{eq:19} and Proposition \ref{cor:flex},
$\rho_{V}$ defines the flex locus of $V$ in the open subset
$\P^n\setminus Z(\ell)$. Varying $\ell$, we deduce that $\rho_{V}$ defines the
flex locus in the whole of $V$.

The resultant $\Res^{\bfy}_{(1,\dots, n,1)}$ is a multihomogeneous
polynomial and, for $i=0,\dots, n-1$, its degree in the set of
variables $\bfc_{i}$ corresponding to the coefficients of the $i$th
polynomial is $n!/(i+1)$. Hence
\begin{equation*}
  \deg_{\bfx}(R_{V,\ell})=\sum_{i=0}^{n-1}   \deg_{\bfx}(f_{V,i+1})
  \deg_{\bfc_{i}}\big(\Res^{\bfy}_{(1,\dots, n,1)}\big)
  = \sum_{k=1}^{n}(d-k)\frac{n!}{k} = d\sum_{k=1}^{n}\frac{n!}{k}- n\cdot n!
\end{equation*}
Hence $ \deg_{\bfx}(\rho_{V}) = \deg_{\bfx}(R_{V,\ell})-n!=
d\sum_{k=1}^{n}\frac{n!}{k}-(n+1)!$, as stated. 

The uniqueness of the polynomial $\rho_{V}$ satisfying \eqref{eq:19}
follows by considering any linear form $\ell$ that is not a zero
divisor modulo $f_{V}$, completing the proof.
\end{proof}

\begin{definition}
  \label{def:3}
  The \emph{flex scheme} of $V$, denoted by $\Flex(V)$, is the
  subscheme of $\P^{n}$ defined by the homogeneous polynomials $f_{V}$ and $\rho_{V}$.
\end{definition}

This scheme does not depend on the choice of $f_{V}$, unique up to a
nonzero scalar factor, nor on that of $\rho_{V}$, unique modulo $f_{V}$. By
Theorem \ref{thm:4}, its support $|\Flex(V)|$, that is, its set of
closed points, coincides with the flex locus of $V$.

\begin{example}\label{ex:curve}
  Let $C$ be a plane curve of degree $d\ge 2$, and
  $f_{C}\in K[x_{0},x_{1},x_{2}]$ its defining polynomial. A
  computation using the Euler identities shows that, for any linear
  form $\ell$, 
\begin{equation}\label{identiti}
  -(d-1)^{2}\Res^{\bfy}_{(1,2,1)}(f_{C,1}(\bfx,\bfy),f_{C,2}(\bfx,\bfy),\ell (\bfy))\equiv 
  \ell^2 \det(\Hess(f_{C})) \mod{f_{C}},
\end{equation}
where $\Hess(f_{C})$ stands for the Hessian matrix of $f_{C}$.  Thus
we recover from Theorem \ref{thm:4} the well-known fact that a point
$p\in C$ is an inflexion point if and only the determinant of the
Hessian matrix of
$f_{C}$ vanishes at $p$, see for instance \cite[\S 7.3,
Theorem~1]{BK86}.  
\end{example}

Giving a closed form for a canonical representative for $\rho_{V}$
modulo $f_{V}$ seems to be a challenge on its own. In the case of
curves, we have just seen in Example \ref{ex:curve} that such a representative is given by the determinant of the Hessian
matrix of $f_{V}$. For $n=3$, Salmon also
obtained a representative of this polynomial as a determinantal closed
formula in terms of covariants, based on an approach by Clebsch
\cite[Articles 589 to 597]{Salmon:tagtd}. It would be interesting to
generalize these formulae to higher dimensions.

\medskip

For a surface $S$ in $\P^{3}$, Theorem \ref{thm:4} shows that the flex
locus of $S$ is defined by an equation of degree
\begin{displaymath}
  \deg(\rho_{S})= d\sum_{k=1}^{3}\frac{3!}{k}-(3+1)!=11 d-24,
\end{displaymath}
recovering the result of Salmon.

\begin{remark}
  \label{rem:2}
  In the book \cite{Salmon:tagtd}, Salmon studied the flex locus of
  the surface $S$ by means of elimination theory.  His Article 473 in
  pages 94--95 of \emph{loc.~cit.}~gives a general method to compute,
  for three surfaces depending on parameters and satisfying a certain
  intersection theoretic condition, the degree of the condition so
  that these surfaces contain a common line.  His Article 588 in pages
  277--278 of \emph{loc.~cit.} applies this degree computation to the
  three surfaces that arise in the study of the flex locus. In our
  notation, these three surfaces are those defining the variety
  $Z_{p}^{3}$ for a point $p\in S$.
\end{remark}

\section{The flex subscheme of a generic hypersurface} \label{section4}

In this section, we show that for
a generic hypersurface of $\P^{n}$ of degree $d\ge n$, the bounds for
the flex locus in Corollary \ref{cor:1} are sharp, or equivalently that the flex scheme is reduced, and hence that it is equal to the flex locus. The next result
corresponds to Theorem \ref{tmain2}\eqref{item:3} in the introduction.

\begin{theorem}\label{mtt}
  Let $d\ge n $ and $f\in K[\bfx]$ a generic homogeneous polynomial of degree
  $d$. Then $\Flex(Z(f))$ is a reduced subscheme of $Z(f)$ of
  dimension $n-2$.
  In particular
  \begin{enumerate}
  \item \label{item:7} $Z(f)$ has no ruled components;
\item \label{item:8} the flex  locus of $Z(f)$ is  the complete intersection of two hypersurfaces of respective degrees
  $d$ and $d\sum_{k=1}^{n}\frac{n!}{k}-(n+1))!$;
\item \label{item:9}  the degree of the flex  locus of $Z(f)$ is equal to $ d^2\sum_{k=1}^{n}\frac{n!}{k}-d(n+1)!$.
  \end{enumerate}
\end{theorem}

Let $d\ge n$ and consider the general polynomial of degree $d$ in the
variables $\bfx$
\begin{equation*}
F=\sum_{|\bfa|=d} c_{\bfa}\bfx^{\bfa} ,
\end{equation*}
the sum being over the vectors $\bfa\in \N^{n+1}$ of length $d$. 
Put
$\bfc=\{c_{\alpha}\}_{|\alpha|=d}$ for the set of ${n+d \choose n}$
variables corresponding to the coefficients of $F$. Thus, $F$ is an irreducible polynomial in $K[\bfc,\bfx]$, bihomogeneous of bidegree $(1,d)$. 

The polynomials $F_{k} \in K[\bfc,\bfx,\bfy]$, $k=0,\dots, d$, are
determined by the expansion
\begin{equation}
  \label{eq:24}
  F(\bfx+t\bfy)=\sum_{k=0}^{d}F_{k}(\bfx,\bfy) \frac{t^{k}}{k!}.
\end{equation}
Following \eqref{eq:2}, for a linear form $\ell\in K[\bfx]$ we set
\begin{equation*}
R_{F,\ell} :=R_{Z(F),\ell} =\Res_{1,\dots,n,1}^{\bfy}(F_1(\bfx,\bfy),\ldots,F_n(\bfx,\bfy),\ell(\bfy)).
\end{equation*}
It is a bihomogeneous polynomial in $K[\bfc,\bfx]$ with bidegree given
by
  \begin{equation}
\label{eq:26}
      \deg_{\bfc}(R_{F,\ell}) =  \sum_{k=1}^{n}\frac{n!}{k}\and
      \deg_{\bfx}(R_{F,\ell}) =  d\sum_{k=1}^{n}\frac{n!}{k}- n \cdot n!.
  \end{equation}

We first prove the existence of a universal polynomial  $\Phi_{d}$
in $ K[\bfc,\bfx]$ with the property that, for any hypersurface $V$ of $\P^{n}$ of
degree $d$, its flex polynomial $\rho_{V}$ can be obtained as the
evaluation of $\Phi_{d}$ at the coefficients of a defining
polynomial of $V$.

\begin{proposition}
\label{prop:6}
  There is a bihomogeneous polynomial $\Phi_{d}\in K[\bfc,\bfx]$ with
  \begin{equation}
\label{eq:23}
      \deg_{\bfc}(\Phi_{d}) =  \sum_{k=1}^{n}\frac{n!}{k}\and
      \deg_{\bfx}(\Phi_{d}) =  d\sum_{k=1}^{n}\frac{n!}{k}-(n+1)!
  \end{equation}
such that, for any squarefree homogeneous polynomial $f\in K[\bfx]$ of degree
$d$,
\begin{equation}
\label{eq:22}
  \rho_{Z(f)}(\bfx)=\Phi_{d}(f,\bfx).
\end{equation}
It is uniquely determined modulo $F$ in the ring $K[\bfc,\bfx]$ by the
condition that, for any linear form $\ell \in K[\bfx]$,
  \begin{equation} \label{eq:191} 
R_{F,\ell} \equiv \ell^{n! }    \Phi_{d} \mod{F} .
\end{equation}
\end{proposition}

\begin{proof}
  Adapting the proof of Theorem \ref{thm:4} to the present situation,
  we can show the existence of a bihomogeneous polynomial
  $\Phi_{d}\in K[\bfc,\bfx]$ satisfying the congruence \eqref{eq:191}
  for any linear form $\ell \in K[\bfx]$.  The formulae \eqref{eq:23}
  for the degrees of $\Phi_{d}$ in the variables~$\bfc$ and $\bfx$
  follow from this congruence and the corresponding formulae for
  $R_{F,\ell}$ in~\eqref{eq:26}.

For a squarefree homogeneous polynomial $f\in K[\bfx]$ of degree $d$,
the congruence \eqref{eq:191} can be evaluated into the coefficients of $f$,
specializing to
\begin{displaymath}
R_{Z(f),\ell} \equiv \ell^{n! }    \Phi_{d}(f)  \mod{f} .
\end{displaymath}
The equality \eqref{eq:22} then follows from the unicity of
$\rho_{Z(f)}$ modulo $f$.
\end{proof}

\begin{remark}
	By the definition of the resultant, as recalled in Section \ref{section2}, the universal polynomial $\Phi_d$ can be chosen as a primitive polynomial with integer coefficients, that is as an irreducible polynomial in $\Z[\bfc,\bfx]$.
\end{remark}

\begin{lemma}\label{lemm:4}
  The polynomial $R_{F,y_{0}}(1,0,\dots, 0)$ is irreducible in
  $K[\bfc]$.
\end{lemma}

\begin{proof}
  Set  for short $R=R_{F,y_{0}}$.  By Proposition \ref{prop:2},
\begin{equation}\label{eq:28}
  R=\Res_{1,\dots, n}^{\bfy'}(F_1(\bfx,0,\bfy'),\ldots,F_n(\bfx,0,\bfy'))
\end{equation}
where $\bfy'$ denotes the set of variables $\{y_{1},\dots, y_{n}\}$.
Hence
\begin{equation*}
R (1,0,\dots, 0)=  \Res_{1,\dots, n}^{\bfy'}(F_1((1,0,\dots, 0),
  (0,\bfy')),\ldots,F_n((1,0,\dots, 0), (0,\bfy')))
\end{equation*}
and, for $j=0,\dots, d$, 
 \begin{equation}\label{eq:29}
 F_{j}((1,0,\dots, 0),  (0,\bfy'))=j!  \sum_{\bfa'}c_{d-j,\bfa'}\,
y_1^{a'_1}\dots y_n^{a'_n} \in  K[\bfc,\bfy'],
 \end{equation}
 the sum being over the vectors $\bfa'\in\N^{n}$ of length $j$.  We
 deduce that $R (1,0,\dots, 0)$ coincides, up to a nonzero scalar,
 with the resultant of $n$ generic polynomials in $n$ variables of
 degrees $1, 2, \dots, n$. In particular, it is irreducible.
\end{proof}

\begin{lemma}\label{cruc}
  The polynomial $R_{F,y_{0}}$ does not depend on the variable
  $c_{d,0,\dots, 0}$, and it is irreducible in
  $K[\bfc, \bfx]_{x_{0}}$.
\end{lemma}

\begin{proof}
The first statement follows from the formula \eqref{eq:28} and the
fact that the polynomials $F_k(\bfx,0,\bfy')$, $k=1,\dots, n$, do not depend on
$c_{d,0,\dots, 0}$ and so neither does $R$, which gives the first
statement.

To prove the second one, set again $R=R_{F,y_{0}}$ and consider a
factorization
\begin{equation*} 
  R =Q_{1}Q_{2}
\end{equation*}
with $Q_{1},Q_{2}\in K[\bfc,\bfx]_{x_{0}}$.  Since $R$ is a
bihomogeneous polynomial in $K[\bfc,\bfx]$, we can assume that its
factors are also of this kind. By Lemma \ref{lemm:4},
$R (1,0,\dots, 0)$ is an irreducible polynomial in $K[\bfc]$ and so
one of these factors, say $Q_{1}$, has degree 0 in the variables
$\bfc$ or equivalently, does not depend on the coefficients of $F$.

For each choice of $p\in \P^{n}\setminus Z(x_{0})$, we can construct a
squarefree homogeneous polynomial $f$ of degree $d$ such that $p$ is
not a flex point of $Z(f)$, and a linear form $\ell$ such that
$\ell(p)\ne 0$.  Proposition \ref{cor:flex} then implies that
$R(p)\ne 0$ and, \emph{a fortiori}, $Q_{1}(p)\ne 0$. Hence $Q_{1}$ is
a unit of $ K[\bfc,\bfx]_{x_{0}}$ and $R$ is irreducible, concluding
the proof.
\end{proof}

\begin{lemma} \label{lemm:1} 
The ideal
  $(F,\Phi_{d})\subset K[\bfc,\bfx]$ is of height 2, and $x_{0}$ is
  not a zero divisor modulo this ideal. 
\end{lemma}

\begin{proof}
  Set again $R=R_{F,y_{0}}$ for short.  By Lemma \ref{cruc}, this
  polynomial does not depend on the variable $c_{d,0,\dots,0}$. Hence,
  it is coprime with $F$, as $F$ is irreducible.  By Proposition \ref{prop:6},
  $ R\equiv x_{0}^{n! } \Phi_{d} \mod{F}$, and so $\Phi_{d}$ is also
  coprime with $F$, giving the first statement.

  For the second statement, set $ F'= F(0,x_{1},\dots, x_{n})$ and
  $\Phi_{d}'=\Phi_{d}(0,x_{1},\dots, x_{n})$, so that
\begin{displaymath}
  F\equiv F' \quad \text{ and } \quad \Phi_{d}\equiv \Phi_{d}' \mod{x_{0}}.
\end{displaymath}
Again by Proposition \ref{prop:6},
\begin{displaymath}
R_{F,y_{n}}(0,x_{1},\dots, x_{n})\equiv x_{n}^{n! } \Phi_{d}' \mod{F'}.  
\end{displaymath}
With the same arguments as for the previous case, we deduce that $F'$
and $\Phi_{d}'$ are coprime. Hence $x_{0},F,P$ is a regular sequence
in $ K[\bfc,\bfx]$.

Since $F,\Phi_{d}$ is a regular sequence in $ K[\bfc,\bfx]$ and this
ring is Cohen-Macaulay, the associated primes of the ideal
$(F,\Phi_{d})$ are of height $2$. Since $x_{0},F,\Phi_{d}$ is also a
regular sequence, $x_{0}$ does not lie in any of these associated
primes and so this variable is not a zero divisor modulo
$(F,\Phi_{d})$, as stated.
\end{proof}

\begin{lemma}
  \label{lemm:3}
  The ideal $(F,\Phi_d)\subset K[\bfc,\bfx]$ is prime.
\end{lemma}

\begin{proof}
  By Lemma \ref{lemm:1}, $x_{0} $ is not a zero divisor modulo
  $(F,\Phi_{d})$ and so the morphism
  \begin{displaymath}
    K[\bfc,\bfx]/(F,\Phi_{d}) \longrightarrow
  K[\bfc,\bfx]_{x_{0}}/(F,\Phi_{d})
  \end{displaymath}
  is an inclusion. Hence, it is enough to prove that the ideal
  $(F,\Phi_{d}) \subset K[\bfc,\bfx]_{x_{0}}$ is prime. Thanks to (\ref{eq:191})  applied with $\ell=x_0$, we obtain an isomorphism
  \begin{displaymath}
    K[\bfc,\bfx]_{x_0}/(F,\Phi_{d}) \longrightarrow
  K[\bfc,\bfx]_{x_{0}}/(F,R_{F,y_0})
  \end{displaymath}
  and we are reduced to show that $(F,R_{F,y_0}) \subset K[\bfc,\bfx]_{x_{0}}$ is prime. 

  Set $\bfc'=\bfc\setminus \{c_{d,0,\dots, 0}\}$ and write
  $F=c_{d,0,\dots, 0} \, x_{0}^{d} +\wt F$ with
  $\wt F\in K[\bfc',\bfx]$. As $R_{F,y_0}$ does not depend on $c_{d,0,\dots, 0}$ by Lemma \ref{cruc}, we get a well-defined isomorphism
\begin{displaymath}
   K[\bfc',\bfx]_{x_{0}}/(R_{F,y_0}) \longrightarrow  K[\bfc,\bfx]_{x_{0}}/(F,R_{F,y_0}).
\end{displaymath}
By Lemma \ref{cruc} again, $R_{F,y_0}$ is irreducible in
$ K[\bfc',\bfx]_{x_{0}}$, and the statement follows. 
\end{proof}

\begin{proof}[Proof of Theorem \ref{mtt}]
  Setting $N={d+n\choose n}-1$, let $Y$ be the subscheme of
  $\P^N\times\P^n$ defined by $F$ and $\Phi_{d}$.  By Lemmas
  \ref{lemm:1} and \ref{lemm:3}, this is an irreducible variety of
  dimension $N+n-2$. Let 
  \begin{displaymath}
  \pi\colon Y\longrightarrow \P^{N}  
  \end{displaymath}
  the map induced by the projection onto the second factor.

  For a generic choice of $\bfalpha\in \P^{N}$, the homogeneous
  polynomial $F(\bfalpha,\bfx)\in K[\bfx]$ is squarefree and, by
  Proposition~\ref{prop:6} and Theorem~\ref{thm:4}, its fiber
  $\pi^{-1}(\bfalpha)$ identifies with the flex locus of the
  hypersurface of $\P^{n}$ defined by this polynomial. The same result
  implies that the dimension of this flex locus is either $n-1 $ or
  $n-2$. Since $Y$ has dimension $N+n-2$, the theorem of dimension of
  fibers implies that $\pi^{-1}(\bfalpha)$ has dimension $n-2$.

Finally, the fact that $Y$ is a variety and Bertini's theorem
\cite[Th\'eor\`eme 6.3(3)]{Jou83} imply that this fiber is reduced,
completing the proof.
\end{proof}

\section{Generic flex points} \label{section6}

For a squarefree homogeneous polynomial $f\in K[\bfx]$
of degree $d\ge n$ and a flex point $p$ of the hypersurface $Z(f)$, we
consider the following properties: 
\begin{enumerate}
\item \label{item:5}  there is a unique
  flex line of $Z(f)$ at $p$;
\item \label{item:6} for a flex line $L$ of $Z(f)$ at $p$, if $d=n$,
  then $L$ is contained in $Z(f)$ whereas if $d>n$, then the order of
  contact of $L$ with $Z(f)$ at $p$ is equal to $n+1$.
\end{enumerate}
In this section we prove the next result, corresponding to Theorem
\ref{tmain2}\eqref{item:4} stated in the introduction. 

\begin{theorem}
  \label{thm:1}
  Let $f\in K[\bfx]$ be a generic homogeneous polynomial of degree
  $d\ge n$, and $p$ a generic point of $\Flex(Z(f))$. Then $(f,p)$
  satisfies the conditions \eqref{item:5} and \eqref{item:6}.
\end{theorem}

We begin with some notation and preliminary
results.  For $d\ge 0$, set $N={d+n\choose n}-1$ and let $\P^N$ be the
projective space of nonzero homogeneous forms of degree $d$ modulo
scalar factors.  
For $k=0,\dots, d$, we introduce the incidence subvariety
{  \begin{align*}
    \Gamma_{k}&=Z(F_{0},\dots, F_{k}) \\
&    =\{ ((c_{\bfa})_{|\bfa|=d},p,q)  \mid F_i(p,q)=0 \text{ for }
    i=0,\dots, n \} 
\subset \P^N\times (\P^{n}\setminus
Z(x_{0})) \times Z(x_{0}),
\end{align*}
  with $Z(x_{0})$ the hyperplane at infinity of $\P^{n}$ and
  $F_{i}$ as in \eqref{eq:24}.}

  \begin{lemma}
    \label{lemm:5}
    The subvariety $\Gamma_{k}$ is irreducible and has dimension
    $N+2n-k$.
  \end{lemma}

  \begin{proof}
    Consider the surjective map
    $\pr_{1}\colon \Gamma_{k}\to (\P^{n}\setminus Z(x_{0})) \times
    Z(x_{0})$
    induced by the projection onto the last two factors.  To study the
    fibers of this map over a point $(p,q)$, we can reduce to the case
     $p=(1:0:\dots:0)$, by applying a suitable linear change of
    coordinates.

 For a point $q\in Z(x_{0})$, the identities in \eqref{eq:29} imply
    that $F_{j}((1,0,\ldots,0),q)$, $j=0,\dots, k$, are nonzero linear
    forms in the variables $\bfc$ depending on disjoint subsets of
    variables, and so they are independent.  Hence the fiber
    $\pr_{1}^{-1}((1,0,\ldots,0),q)$ is a linear space of dimension
    $N-k$, and a similar statement holds for any pair of points
    $(p,q)$.
    Thus $\Gamma_{k}$ is a geometric vector bundle of dimension
    $N-k-1$ over the base space
    $(\P^{n}\setminus Z(x_{0})) \times Z(x_{0})$. Since this base is
    irreducible and has dimension $2n-1$, the subvariety is also
    irreducible and has dimension $N+2n-k$, as claimed.
  \end{proof}

  For the special case $k=n$, the subvariety $\Gamma_{n}$ consists of
  the triples $(f,p,q)$ where $f$ is a homogeneous polynomial of
  degree $d$, $p\in \P^{n}\setminus Z(x_{0})$ is a flex point of the
  hypersurface $Z(f)$, and $q\in Z(x_{0})$ determines a flex line
{  containing} $p$.  Let
  $\Omega\subset \P^{N}\times (\P^{n}\setminus Z(x_{0}))$ denotes the set
  of pairs $(f,p)$ where $p \in \P^{n}\setminus Z(x_{0})$ is a flex
  point of $Z(f)$, and
\begin{equation}
\label{eq:27}
  \pi\colon \Gamma_{n}\longrightarrow \Omega
\end{equation}
 the map induced by the projection of 
$\P^N\times (\P^{n}\setminus
Z(x_{0})) \times Z(x_{0})$ onto its
  first two factors.  

\begin{proposition}
  \label{prop:3}
  The map $\pi$ is birational.
\end{proposition}

\begin{proof}
  Since $\pi$ is the restriction   to the irreducible subvariety
  $\Gamma_{n}$ of the proper map
  \begin{math}
  \P^N\times (\P^{n}\setminus Z(x_{0})) \times Z(x_{0})\longrightarrow \P^N\times
  (\P^{n}\setminus Z(x_{0}))   
\end{math}, its image $\Omega$ is also an irreducible subvariety.
Indeed, by Proposition~\ref{cor:flex} and Proposition~\ref{prop:6}, it is the subvariety of
$\P^{N}\times (\P^{n}\setminus Z(x_{0}))$ defined by the polynomials
$F$ and $R_{F,y_{0}}.$

The inversion property of the resultant (Proposition \ref{prop:9}),
implies that the map $\pi$ is invertible on the open subset of points
$(f,p)\in \Omega$ where
  \begin{equation}\label{eq:8}
    \frac{\partial \Res_{(1,\dots, n,1)}^{\bfy}}{\partial
      c_{i_{0},\bfa_{0}}}(f_{1}(\bfp,\bfy), \dots, f_{n}(\bfp,\bfy),
    y_{0})\ne 0
  \end{equation}
  for a representative $\bfp\in K^{n+1}\setminus \{\bfzero\}$ of $p$
  and a pair of indices $0\le i_{0}\le n-1$ and $\bfa_{0}\in \N^{n+1}$
  with $|\bfa_{0}|=i_{0}+1$.

  We next want to prove that this open subset is nonempty. To this
  end, it is enough to show that there is a point $(f,p_{0})$ in
  $\Omega$ with $p_{0}=(1:0\dots:0)\in \P^{n}\setminus Z(x_{0})$
  satisfying at least one of the inequations \eqref{eq:8}.
  We have that $F(1,0,\dots,0)= c_{d,0,\dots,0}$ and, by Lemma
  \ref{cruc}, the polynomial $R_{F,y_{0}}$ does not depend on this
  variable. Thus $(f,p_{0})\in \Omega$ if and only if it satisfies the
  independent conditions
\begin{equation}
  \label{eq:20}
  c_{d,0,\dots,0}=0 \and    R_{Z(f),y_{0}} (1,0,\dots,0) =0.
\end{equation}
By  \eqref{eq:29}, each of the evaluations in \eqref{eq:8} for
the point $(f,p_{0})$ coincide, up to a fixed (that is, neither
depending on $i_{0}$ nor on $\bfa_{0}$) nonzero scalar factor, with
\begin{equation*}
    \frac{\partial  R_{Z(f),y_{0}}}{\partial c_{\bfb_{0}}}
    (1,0,\dots,0)
\end{equation*}
for a vector $\bfb_{0}\in \N^{n+1}$ with $|\bfb_{0}|=d$.  Hence, the
condition that $(f,p_{0})$ satisfies \eqref{eq:8} is equivalent to
\begin{equation}
  \label{eq:21}
      \nabla R_{F,y_{0}}(1,0,\dots,0)(f) \neq{\bf 0},
\end{equation}
where $\nabla R_{F,y_{0}}$ denotes the gradient operator. By Lemma
\ref{lemm:4}, $R_{F,y_{0}}$ is an irreducible polynomial, and
\emph{a fortiori} squarefree.  Hence, the condition \eqref{eq:21}
is verified for a generic $f$ satisfying \eqref{eq:20}. 

We deduce that the map $\pi$ is invertible on a nonempty open subset
of $\Omega$. Since $\Omega $ is irreducible, such an open subset is
dense, and so $\pi$ is birational.
\end{proof}

\begin{proof}[Proof of Theorem \ref{thm:1}]
  By Proposition \ref{prop:3}, there are dense open subsets
  $U\subset \Gamma_{n}$ and $W\subset \Omega$ such that the
  restriction of the map $\pi$ in \eqref{eq:27} to these subsets is an
  isomorphism. In particular, for each $(f,p)\in W$ there is a unique
  flex line {containing} the point $p$.

  If $d=n$, then any such flex line has order of contact at least
  $n+1$ at the point $p$, and so it is necessarily contained in
  $Z(f)$. 
  
  If $d>n$ then $\Gamma_{n+1}$ is a proper subvariety of $\Gamma_{n}$
  by Lemma \ref{lemm:5}. By Lemma \ref{lmain}\eqref{item:1}, for each
  $(f,p,q)\in U\setminus \Gamma_{n+1}$, the line {containing}  $p$ and
  $q$ has order of contact equal to $n+1$. Hence, every pair $(f,p)$
  in the dense open subset $W':=W\setminus \pi(\Gamma_{n+1})$ of
  $\Omega$ satisfies both conditions \eqref{item:5} and
  \eqref{item:6}.

  Set $Z=\Omega\setminus W'$ and consider the map
  $\varpi\colon Z\to \P^{N}$ defined by $(f,p)\mapsto f$. If this map
  is not dominant, then for $f\in \P^{N}\setminus \varpi(Z)$ we have
  that $\{f\}\times \Flex(Z(f))$ is disjoint from $Z$, giving the
  statement in this case.

 Otherwise, by the theorem of dimension of fibers \cite[\S1.6,
 Theorem 7]{Shafarevich:bag}, there is a dense open subset
 $T\subset \P^{N}$ such that, for $f\in T$,
  \begin{equation*}
    \dim(\varpi^{-1}(f))= \dim(Z)-\dim(\P^{N}) <n-2.
  \end{equation*}
  On the other hand, $\dim(\Flex(Z(f))) $ is either $n-1$ or $n-2$.
  Hence for all $f\in T$, no component of $\{f\}\times \Flex(Z(f))$
  can be contained in $Z$. 

  In both cases, there is a dense open subset $T$ of $\P^{N}$ such
  that, for each $f\in T$, we have that $f$ is squarefree and there is
  a dense open subset $U_{f}$ of the flex locus of $Z(f)$ such that
  for each $p\in U_{f}$, the pair $(f,p)$ satisfies the conditions
  \eqref{item:5} and \eqref{item:6}, completing the proof.
\end{proof}

% \bibliographystyle{amsalpha} 
% \bibliography{biblio}

\begin{thebibliography}{JKSS04}

\bibitem[BK86]{BK86}
E.~Brieskorn and H.~Kn\"orrer, \emph{Plane algebraic curves}, Birkh\"auser,
  1986.

\bibitem[CLO05]{CLO98}
D.~A. Cox, J.~B. Little, and D.~O'Shea, \emph{Using algebraic geometry}, second
  ed., Grad. Texts in Math., vol. 185, Springer, 2005.

\bibitem[EH16]{EH}
D.~Eisenbud and J.~Harris, \emph{3264 and all that---a second course in
  algebraic geometry}, Cambridge Univ. Press, 2016.

\bibitem[GK15]{GuthKatz:eddpp}
L.~Guth and N.~H. Katz, \emph{On the {E}rd\"os distinct distances problem in
  the plane}, Ann. of Math. (2) \textbf{181} (2015), 155--190.

\bibitem[GKZ94]{GKZ94}
I.~M. Gelfand, M.~M. Kapranov, and A.~V. Zelevinsky, \emph{Discriminants,
  resultants, and multidimensional determinants}, Math. Theory Appl.,
  Birkh\"auser, 1994.

\bibitem[GZ18]{GZ18} L. Guth and J. Zahl, \emph{Algebraic curves, rich
    points, and doubly-ruled surfaces}, e-print arXiv:1503.02173, to
  appear in Amer. J. Math..

\bibitem[JKSS04]{JeronimoKrickSabiaSombra:cccf}
G.~Jeronimo, T.~Krick, J.~Sabia, and M.~Sombra, \emph{The computational
  complexity of the {C}how form}, Found. Comput. Math. \textbf{4} (2004),
  41--117.

\bibitem[Jou83]{Jou83}
J.-P. Jouanolou, \emph{Th\'eor\`emes de {B}ertini et applications}, Progr.
  Math., vol.~42, Birkh\"auser, 1983.

\bibitem[Jou91]{Jou}
\bysame, \emph{Le formalisme du r\'esultant}, Adv. Math. \textbf{90} (1991),
  117--263.

\bibitem[Kat14]{Kat14}
N.~Katz, \emph{The flecnode polynomial: a central object in incidence
  geometry}, Proc. ICM 2014, vol. III, 2014, pp.~303--314.

\bibitem[Kol15]{Kollar:stttd3}
J.~Koll\'ar, \emph{Szemer\'edi-{T}rotter-type theorems in dimension 3}, Adv.
  Math. \textbf{271} (2015), 30--61.

\bibitem[Lan99]{Land}
J.~M. Landsberg, \emph{Is a linear space contained in a submanifold? {O}n the
  number of derivatives needed to tell}, J. Reine Angew. Math. \textbf{508}
  (1999), 53--60.

\bibitem[Sal65]{Salmon:tagtd}
G.~Salmon, \emph{A treatise on the analytic geometry of three dimensions.
  {V}ol. {II}}, 1865, reprinted fifth edition at Chelsea Publishing Co., 1965.

\bibitem[Sha94]{Shafarevich:bag}
I.~R. Shafarevich, \emph{Basic algebraic geometry. 1. {V}arieties in projective
  space}, second ed., Springer-Verlag, 1994.

\bibitem[SS18]{SharirSolomon:ibplttdv}
M.~Sharir and N.~Solomon, \emph{Incidences between points and lines on two- and
  three-dimensional varieties}, Discrete Comput. Geom. \textbf{59} (2018),
  88--130.

\bibitem[Tao14]{Tao:MCStcdg}
T.~Tao, \emph{The {M}onge-{C}ayley-{S}almon theorem
  via~classical~differential~geometry}, blog entry at
  \url{https://terrytao.wordpress.com/2014/03/28/}, 2014.

\end{thebibliography}

\newcommand{\noopsort}[1]{} \newcommand{\printfirst}[2]{#1}
  \newcommand{\singleletter}[1]{#1} \newcommand{\switchargs}[2]{#2#1}
  \def\cprime{$'$}
\providecommand{\bysame}{\leavevmode\hbox to3em{\hrulefill}\thinspace}
\providecommand{\MR}{\relax\ifhmode\unskip\space\fi MR }
\providecommand{\MRhref}[2]{%
  \href{http://www.ams.org/mathscinet-getitem?mr=#1}{#2}
}
\providecommand{\href}[2]{#2}

\end{document}